\documentclass[12pt,reqno]{amsart}
\usepackage{times}

\newtheorem{thm}{Theorem}
\newtheorem{cor}[thm]{Corollary}
\newtheorem{lem}[thm]{Lemma}

\newtheorem*{thmm}{Main Theorem}

\newcommand{\C}{{\mathbb C}}
\newcommand{\cn}{{\C^n}}
\newcommand{\dla}{d\lambda_\alpha}
\newcommand{\lta}{L^2(\cn,\dla)}
\newcommand{\fta}{F^2_\alpha}

\newcommand{\incn}{\int_\cn}
\newcommand{\e}{\varepsilon}
\newcommand{\re}{{\rm Re}}

\begin{document}

\title[Toeplitz products]{Products of Toeplitz operators\\
on the Fock space}

\author{Hong Rae Cho, Jong-Do Park, and Kehe Zhu}

\address{Cho: Department of Mathematics, Pusan National University, Pusan 609-735, Korea.}
\email{chohr@pusan.ac.kr}
\address{Park: School of Mathematics, KIAS, Hoegiro 87, Dongdaemun-gu, Seoul 130-722,  Korea.}
\email{jdpark@kias.re.kr}
\address{Zhu: Department of Mathematics and Statistics, SUNY, Albany, NY 12222, USA.}
\email{kzhu@math.albany.edu}

\begin{abstract}
Let $f$ and $g$ be functions, not identically zero, in the Fock space $\fta$ of $\cn$. 
We show that the product $T_fT_{\overline g}$ of Toeplitz operators on $\fta$ is bounded if and 
only if $f(z)=e^{q(z)}$ and $g(z)=ce^{-q(z)}$, where $c$ is a nonzero constant and $q$ is a linear polynomial.
\end{abstract}

\thanks{This work was supported by the National Research Foundation of Korea(NRF) grant funded by the Korea government(MEST) (NRF-2011-0013740 for the first author) and (NRF-2010-0011841 for the second author).}

\maketitle

\section{Introduction}

Let $\cn$ be the complex $n$-space. For points $z=(z_1,\cdots,z_n)$ and 
$w=(w_1,\cdots,w_n)$ in $\cn$ we write
$$z\cdot\overline w=\sum_{j=1}^nz_j\overline{w}_j,\qquad
|z|=\sqrt{z\cdot\overline z}.$$

Let $dv$ be ordinary volume measure on $\cn$. For any positive parameter $\alpha$ we
consider the Gaussian measure
$$\dla(z)=\left(\frac\alpha\pi\right)^ne^{-\alpha|z|^2}\,dv(z).$$
The Fock space $\fta$ is the closed subspace of entire functions in 
$\lta$. The orthogonal projection $P:\lta\to\fta$ is given by
$$Pf(z)=\incn K(z,w)f(w)\,\dla(w),$$
where $K(z,w)=e^{\alpha z\cdot\overline w}$ is the reproducing kernel of $\fta$.

We say that $f$ satisfies Condition (G) if the function
$z\mapsto f(z)e^{\alpha z\cdot\overline w}$
belongs to $L^1(\cn,\dla)$ for every $w\in\cn$. Equivalently, $f$ satisfies 
Condition (G) if every translate of $f$, $z\mapsto f(z+w)$, belongs to $L^1(\cn,\dla)$.
If $f\in\fta$, then there exists a constant $C>0$ such that 
$$|f(z)|\le Ce^{\frac\alpha2|z|^2},\qquad z\in\cn.$$
This clearly implies that $f$ satisfies Condition (G).

If $f$ satisfies Condition (G), we can define a linear operator $T_f$ on $\fta$ by
$T_fg=P(fg)$, where
$$g(z)=\sum_{k=1}^Nc_kK(z,w_k)$$
is any finite linear combination of kernel functions. It is easy to see that the set 
of all finite linear combinations of kernel functions is dense in $\fta$. Here $P(fg)$ 
is to be interpreted as the following integral:
$$T_fg(z)=\incn f(w)g(w)e^{\alpha z\cdot\overline w}\,\dla(w),\qquad z\in\cn.$$
Therefore, for $g$ in a dense subset of $\fta$, $T_fg$ is a well-defined entire function
(not necessarily in $\fta$ though).

The purpose of this article is to study the Toeplitz product $T_fT_{\overline g}$, where
$f$ and $g$ are functions in $\fta$. Such a product is well defined on the set of finite
linear combinations of kernel functions. Our main concern is the following: what conditions
on $f$ and $g$ will ensure that the Toeplitz product $T_fT_{\overline g}$ extends to a bounded
(or compact) operator on $\fta$?

This problem was first raised by Sarason in \cite{S} in the context of Hardy and Bergman
spaces. It was partially solved for Toeplitz operators on the Hardy space of the unit circle 
in \cite{Zh}, on the Bergman space of the unit disk in \cite{SZ99}, on the Bergman space 
of the polydisk in \cite{SZ03}, and on the Bergman space of the unit ball in \cite{P,SZ07}. 
In all these cases, the necessary and/or sufficient condition for $T_fT_{\overline g}$ to 
be bounded is 
$$\sup_{z\in\Omega}\widetilde{|f|^{2+\e}}(z)\widetilde{|g|^{2+\e}}(z)<\infty,$$
where $\e$ is any positive number and $\widetilde f$ denotes the Berezin transform of $f$. 
Note that in the Hardy space case, the Berezin transform is nothing but the classical 
Poisson transform.

Here we obtain a much more explicit characterization for $T_fT_{\overline g}$ to be 
bounded on the Fock space.

\begin{thmm}
Let $f$ and $g$ be functions in $\fta$, not identically zero. Then $T_fT_{\overline g}$ is
bounded on $\fta$ if and only if $f=e^q$ and $g=ce^{-q}$, where $c$ is a nonzero
complex constant and $q$ is a complex linear polynomial.
\end{thmm}

Furthermore, our proof reveals that when $T_fT_{\overline g}$ is bounded, it must
be a constant times a unitary operator. Consequently, $T_fT_{\overline g}$ is never
compact unless it is the zero operator.

As another by-product of our analysis, we will construct a class of unbounded, densely 
defined, operators on the Fock space whose Berezin transform is bounded. It has been known 
that such operators exist, but our examples are very simple products of Toeplitz operators.

We wish to thank Boorim Choe and Hyungwoon Koo for useful conversations.

\section{Proof of Main Result}

For any point $a\in\cn$ we consider the operator $U_a:\fta\to\fta$ defined by
$$U_af(z)=f(z-a)k_a(z),$$
where
$$k_a(z)=\frac{K(z,a)}{\sqrt{K(a,a)}}=e^{\alpha z\cdot\overline a-\frac\alpha2|a|^2}$$
is the normalized reproducing kernel of $\fta$ at $a$. It follows from a change of variables
that each $U_a$ is a unitary operator on $\fta$.

We begin with the very special case of Toeplitz operators induced by kernel functions.

\begin{lem}
Let $a\in\cn$, $f(z)=e^{\alpha z\cdot\overline a}$, and $g(z)=e^{-\alpha z\cdot\overline a}$. 
We have
$$T_fT_{\overline g}=e^{\frac\alpha2|a|^2}U_a.$$
In particular, $T_fT_{\overline g}$ is bounded on $\fta$.
\label{lem1}
\end{lem}

\begin{proof}
To avoid triviality we assume that $a$ is nonzero. The Toeplitz 
operator $T_f$ is just multiplication by $f$, as a densely defined unbounded linear 
operator. So we focus on the operator $T_{\overline g}$.

Given any function $h\in\fta$, we have
\begin{eqnarray*}
T_{\overline g}h(z)&=&\incn\overline{g(w)}\,h(w)K(z,w)\,\dla(w)\\
&=&\incn h(w)e^{\alpha(z-a)\cdot\overline w}\,\dla(w)\\
&=&h(z-a).
\end{eqnarray*}
Therefore, the Toeplitz operator $T_{\overline g}$ is an operator of translation, and
$$T_fT_{\overline g}h(z)=e^{\alpha z\cdot\overline a}h(z-a)=e^{\frac\alpha2|a|^2}U_ah(z).$$
This proves the desired result.
\end{proof}

An immediate consequence of Lemma~\ref{lem1} is that if $f=C_1e^q$ and $g=C_2e^{-q}$, where
$C_1$ and $C_2$ are complex constants and $q$ is a complex linear polynomial, then there
exists a complex constant $c$ and a unitary operator $U$ such that $T_fT_{\overline g}=cU$.

To deal with more general symbol functions, we need the following characterization of
nonvanishing functions in $\fta$.

\begin{lem}
If $f$ is a nonvanishing function in $\fta$, then there exists a complex 
polynomial $q$, with $\deg(q)\le2$, such that $f=e^q$.
\label{lem2}
\end{lem}

\begin{proof}
In the case when the dimension $n=1$, the Weierstrass factorization of functions in the Fock
space $\fta$ takes the form $f(z)=P(z)e^{q(z)}$, where $P$ is the canonical Weierstrass
product associated to the zero sequence of $f$, and $q(z)=az^2+bz+c$ is a quadratic polynomial
with $|a|<\frac\alpha2$. In particular, if $f$ is zero-free, then $f=e^q$ for some 
quadratic polynomial. See \cite{Z}.

When $n>1$, we no longer have such a nice factorization. But the absence of zeros makes
a special version of the factorization above still valid. More specifically, if $f$ is 
any function in $\fta=F^2_\alpha(\cn)$ and $f$ is nonvanishing, then the function 
$z_1\mapsto f(z_1,\cdots,z_n)$ is in $F^2_\alpha(\C)$, so by the factorization theorem
stated in the previous paragraph,
$$f(z_1,\cdots,z_n)=e^{az_1^2+bz_1+c},$$
where $a$, $b$, and $c$ are holomorphic functions of $z_2,\cdots,z_n$. Repeat this for
every independent variable, we conclude that $f=e^q$ for some polynomial of degree $2n$ 
or less.

Recall that every function $f\in\fta$ satisfies the pointwise estimate
$$|f(z)|\le Ce^{\frac\alpha2|z|^2},\qquad z\in\cn.$$
If $q$ is a polynomial of degree $N$ and $N>2$, then
for any fixed $\zeta=(\zeta_1,\cdots,\zeta_n)$ on the unit sphere of $\cn$ with each
$\zeta_k\not=0$, and for $z=r\zeta$, where $r>0$, we have $q(z)\sim r^N$ as $r\to\infty$, 
which shows that the estimate $|f(z)|\le Ce^{\frac\alpha2|z|^2}$ is impossible to hold. 
This shows that the degree of $q$ is less than or equal to $2$.
\end{proof}

We can now prove the main result, which we restate as follows.

\begin{thm}
Suppose $f$ and $g$ are functions in $\fta$. Then the Toeplitz product $T_fT_{\overline g}$
is bounded on $\fta$ if and only if one of the following two conditions holds:
\begin{enumerate}
\item[(a)] At least one of $f$ and $g$ is identically zero.
\item[(b)] There exists a linear polynomial $q$ and a nonzero constant $c$ such that 
$f=e^q$ and $g=ce^{-q}$.
\end{enumerate}
\label{thm3}
\end{thm}

\begin{proof}
If condition (a) holds, then the Toeplitz product $T_fT_{\overline g}$ is $0$. If 
condition (b) holds, the boundedness of $T_fT_{\overline g}$ follows from Lemma~\ref{lem1}.

Next assume that $T=T_fT_{\overline g}$ is bounded on $\fta$. Then the Berezin transform 
$\widetilde T$ is a bounded function on $\cn$, where
$$\widetilde T(z)=\langle T_fT_{\overline g}k_z,k_z\rangle,\qquad z\in\cn.$$
It follows from the integral representation of $T_{\overline g}$ and the reproducing 
property of the kernel function $e^{\alpha z\cdot\overline w}$ that 
$T_{\overline g}k_z=\overline{g(z)}\,k_z$.
Therefore,
$$\widetilde T(z)=\overline{g(z)}\langle fk_z,k_z\rangle,\qquad z\in\cn.$$
Write the inner product above as an integral and apply the reproducing property of
the kernel function $e^{\alpha z\cdot\overline w}$ one more time. We obtain
$\widetilde T(z)=f(z)\overline{g(z)}$. It follows that $|f(z)g(z)|\le\|T\|$ for all $z\in\cn$.
But $fg$ is entire, so by Liouville theorem, there is a constant $c$ such that $fg=c$.

If $c=0$, then at least one of $f$ and $g$ must be identically zero, so condition (a) holds.

If $c\not=0$, then both $f$ and $g$ are nonvanishing. By Lemma~\ref{lem2}, there exists a
complex polynomial $q$, with $\deg(q)\le2$, such that $f=e^{q}$ and $g=ce^{-q}$. It remains 
for us to show that $\deg(q)\le1$.

Let us assume $\deg(q)=2$, in the hope of reaching a contradition, and write $q=q_2+q_1$, 
where $q_1$ is linear and $q_2$ is a homogeneous polynomial of degree $2$. By the 
boundedness of $T=T_fT_{\overline g}$ on $\fta$, the function
$$T(z,w)=\langle T_fT_{\overline g}k_z,k_w\rangle,\qquad z\in\cn, w\in\cn,$$
is bounded on $\cn\times\cn$. We proceed to show that this is impossible unless $q_2=0$.

Again, by the integral representation for Toeplitz operators and the reproducing property
of the kernel function $e^{\alpha z\cdot\overline w}$, it is easy to obtain that
$$T(z,w)=f(w)\overline{g(z)}e^{-\frac\alpha2|z|^2+\alpha w\cdot\overline z-\frac\alpha2|w|^2}.$$
It follows that
$$|T(z,w)|=|f(w)g(z)|e^{-\frac\alpha2|z-w|^2}$$
for all $(z,w)\in\cn\times\cn$. Using the explict form of $f$ and $g$, we can write
$$|T(z,w)|=|c\exp(q_2(w)-q_2(z)+q_1(w)-q_1(z))|e^{-\frac\alpha2|z-w|^2}.$$
Since $q_1$ is linear, it is easy to see that there is a point $a\in\cn$ such that
$$q_1(w)-q_1(z)=(w-z)\cdot\overline a$$ 
for all $z$ and $w$.

For the second-degree homogeneous polynomial $q_2$ we can find a complex matrix 
$A=A_{n\times n}$, symmetric in the real sense, such that $q_2(z)=\langle Az,z\rangle$, 
where $\langle\ ,\ \rangle$ is the real inner product. Fix two points $u$ and $v$ in $\cn$ 
such that $\re\langle Au,v\rangle\not=0$. This is possible as long as $A\not=0$. Now let 
$z=ru$ and $w=ru+v$, where $r$ is any real number. We have
\begin{eqnarray*}
q_2(w)-q_2(z)&=&q_2(z+v)-q_2(z)\\
&=&\langle A(z+v),z+v\rangle-\langle Az,z\rangle\\
&=&\langle Az,v\rangle+\langle Av,z\rangle+\langle Av,v\rangle\\
&=&2r\langle Au,v\rangle+\langle Av,v\rangle.
\end{eqnarray*}
It follows that there exists a positive constant $M=M(u,v)$ such that
$$|T(z,w)|=M|\exp(2r\langle Au,v\rangle)|=M\exp(2r\re\langle Au,v\rangle).$$
Since $\re\langle Au,v\rangle\not=0$, this shows that $T(z,w)$ cannot be a bounded 
function on $\cn\times\cn$. This contradition shows that $A=0$ and the polynomial 
$q$ must be linear.
\end{proof}

As a consequence of the analysis above, we obtain an interesting class of unbounded operators
on $\fta$ whose Berezin transforms are bounded.

\begin{cor}
Suppose $f(z)=e^q$ and $g=e^{-q}$, where $q$ is any second-degree homogeneous polynomial
whose coefficients are small enough so that $f$ and $g$ belong to $\fta$. Then the Toeplitz
product $T_fT_{\overline g}$ is unbounded on $\fta$, but its Berezin transform is bounded.
\label{cor4}
\end{cor}

\begin{proof}
By Theorem~\ref{thm3}, the operator $T_fT_{\overline g}$ is unbounded. On the other hand, by
the proof of Theorem~\ref{thm3}, the Berezin transform of $T=T_fT_{\overline g}$ is given by
$$\widetilde T(z)=f(z)\overline{g(z)},\qquad z\in\cn.$$
It follows that $|T(z)|=|f(z)g(z)|=1$ for all $z\in\cn$.
\end{proof}

Another consequence of our earlier analysis is the following.

\begin{cor}
If $f$ and $g$ are functions in $\fta$, then the following conditions are equivalent:
\begin{enumerate}
\item[(a)] $T_fT_{\overline g}$ is compact.
\item[(b)] $T_fT_{\overline g}=0$.
\item[(c)] $f=0$ or $g=0$.
\end{enumerate}
\end{cor}

\begin{proof}
Combining Lemma~\ref{lem1} and Theorem~\ref{thm3}, we see that whenever $T_fT_{\overline g}$
is compact on $\fta$, we must have $f=0$ or $g=0$. This clearly gives the desired result.
\end{proof}

\section{Further Remarks}

For any $0<p\le\infty$ let $F^p_\alpha$ denote the Fock space consisting of entire functions
$f$ such that the function $f(z)e^{-\frac\alpha2|z|^2}$ belongs to $L^p(\cn,dv)$. When
$0<p<\infty$, the norm in $F^p_\alpha$ is defined by
$$\|f\|_{p,\alpha}=\left[\left(\frac{p\alpha}{2\pi}\right)^n\incn\left|f(z)e^{-\frac\alpha2
|z|^2}\right|^p\,dv(z)\right]^{\frac1p}.$$
For $p=\infty$, the norm in $F^\infty_\alpha$ is defined by
$$\|f\|_{\infty,\alpha}=\sup_{z\in\cn}|f(z)|e^{-\frac\alpha2|z|^2}.$$

It is easy to check that the normalized reproducing kernel
$$k_a(z)=e^{\alpha z\cdot\overline a-\frac\alpha2|a|^2}$$
is a unit vector in each $F^p_\alpha$, where $0<p\le\infty$. Also, it can be shown that 
the set of functions of the form
$$f(z)=\sum_{k=1}^Nc_kK(z,a_k)=\sum_{k=1}^Nc_ke^{\alpha z\cdot\overline a_k}$$
is dense in each $F^p_\alpha$, where $0<p<\infty$. See \cite{Z}.

Therefore, if $0<p<\infty$ and $f$ satisfies Condition (G), we can consider the action
of the Toeplitz operator $T_f$ (defined using the integral representation in Section 1) on
$F^p_\alpha$. Also, if $f\in F^p_\alpha$, then it satisfies the pointwise estimate
$|f(z)|\le Ce^{\frac\alpha2|z|^2}$, which implies that $f$ satisfies Condition (G).

When $1<p<\infty$ and $1/p+1/q=1$, the dual space of $F^p_\alpha$ can be identified with
$F^q_\alpha$ under the integral pairing
$$\langle f,g\rangle_\alpha=\incn f(z)\overline{g(z)}\,\dla(z).$$
When $0<p\le1$, the dual space of $F^p_\alpha$ can be identified with $F^\infty_\alpha$ under
the same integral pairing above. See \cite{JPR,Z}.

Thus for functions $f$ and $g$ in $F^p_\alpha$, if the Toeplitz product 
$T=T_fT_{\overline g}$ is bounded on $F^p_\alpha$, we can still consider the function
$$T(z,w)=\langle T_fT_{\overline g}k_z,k_w\rangle_\alpha$$
on $\cn\times\cn$. Exactly the same arguments from Section 2 will yield the following result.

\begin{thm}
Suppose $0<p<\infty$. If $f$ and $g$ are functions in $F^p_\alpha$, not identically zero,
then the Toeplitz product $T_fT_{\overline g}$ is bounded on $F^p_\alpha$ if and only if
$f=e^q$ and $g=ce^{-q}$, where $c$ is a nonzero complex contant and $q$ is a complex linear
polynomial.
\end{thm}

It would be nice to extend our results here to more general Fock-type spaces. In particular,
generalization to the Fock-Sobolev spaces studied in \cite{CCK, CZ} should be possible.

It would also be interesting to take a second look at the original Hardy space setting.
More specifically, if $f$ and $g$ are functions in the Hardy space $H^2$ (of the unit disk,
for example), the the boundedness of the Toeplitz product $T_fT_{\overline g}$ on $H^2$
implies that the product function $fg$ is in $H^\infty$. Is it possible to derive more
detailed information about $f$ and $g$, say in terms of inner and outer functions? A more
explicit condition on $f$ and $g$ (as opposed to the condition $\widetilde{|f|^{2+\e}}
\widetilde{|g|^{2+\e}}\in L^\infty$) would certainly be more desirable.

We hope that this paper will generate some further interest in this subject.


\end{document}